\pdfoutput=1
\documentclass[journal,onecolumn,12pt,twoside]{IEEEtranTCOM}
\usepackage{etex}

\normalsize

\usepackage{cite}
\usepackage{setspace}

\usepackage{graphicx}
\usepackage{color}
\usepackage{pst-sigsys}
\ifCLASSINFOpdf
\else
\fi

\usepackage[cmex10]{amsmath}
\usepackage{amsmath,dsfont}
\usepackage{verbatim}
\usepackage{amssymb}
\usepackage{amsthm}
\usepackage{nccmath}
\usepackage{enumerate}
  \usepackage{pgfplots}
  \usepackage[normalem]{ulem}
  \usepackage{cases}
  \usepackage{bbm}  
  \usepackage{mathtools}
\usepackage[final]{hyperref}
  \usepackage{cleveref}
\pgfplotsset{plot coordinates/math parser=false}
\renewcommand{\descriptionlabel}[1]{%
  \hspace\labelsep \upshape{ \textit{\underline{#1:}}}%
}

\newtheorem{mydef}{Definition}

\newtheorem{theorem}{Theorem}

\newtheorem{remark}{Remark}

\newcommand{\chainsumOne}{\ensuremath{\sum_{k_1=1}^{\sizeS} \sum_{k_2 = 1}^{\sizeS} \ldots \sum_{k_{n-2} = 1}^{\sizeS}}}
\newcommand{\chainsum}{\ensuremath{\sum_{k_1=1}^{\sizeS} \sum_{k_2 = 1}^{\sizeS} \ldots \sum_{k_{n-1} = 1}^{\sizeS}}}

\newcommand{\sizeS}{\ensuremath{|\myState|}}
\newcommand{\Si}[1]{\ensuremath{S_{#1}}}

\newcommand{\trans}[2]{\ensuremath{p_{#1#2}}}
\newcommand{\transM}{\ensuremath{P}}
\newcommand{\transMsup}[1]{\ensuremath{\transM^{#1}}}
\newcommand{\transMsupij}[3]{\ensuremath{\transMsup{#1}(#2, #3)}}
\newcommand{\transR}{\ensuremath{R}}
\newcommand{\transQ}{\ensuremath{Q}}

\newcommand{\fundMat}{\ensuremath{N}}

\newcommand{\rew}[2]{\ensuremath{\rho_{#1#2}}}
\newcommand{\Rew}{\ensuremath{\Theta}}

\newcommand{\RewSubTransAbs}{\ensuremath{\Theta_{2}}}
\newcommand{\RewSubTransTrans}{\ensuremath{\Theta_{1}}}

\newcommand{\RewSubCi}[1]{\ensuremath{B_{#1}}}
\newcommand{\RewSubDi}[1]{\ensuremath{A_{#1}}}
\newcommand{\RewSubC}{\ensuremath{\RewSubCi{n}}}
\newcommand{\RewSubD}{\ensuremath{\RewSubDi{n}}}

\newcommand{\HtransRew}{\ensuremath{H}}

\newcommand{\HSubTransAbs}{\ensuremath{H_{2}}}
\newcommand{\HSubTransTrans}{\ensuremath{H_{1}}}
\newcommand{\HtransRewij}[2]{\ensuremath{H(#1, #2)}}

\newcommand{\Rni}[1]{\ensuremath{R_{#1}}}
\newcommand{\Rn}{\ensuremath{\Rni{n}}}

\newcommand{\barR}{\ensuremath{\bar{R}_{n}}}
\newcommand{\barRij}[2]{\ensuremath{\bar{R}_{n}(#1, #2)}}
\newcommand{\barRsubij}[3]{\ensuremath{\bar{R}_{#1}(#2, #3)}}

\newcommand{\barRinf}{\ensuremath{\bar{R}_{\infty}}}
\newcommand{\barRi}[1]{\ensuremath{\bar{R}_{n}(#1)}}

\newcommand{\hatR}{\ensuremath{\hat{R}_{n}}}
\newcommand{\hatRij}[2]{\ensuremath{\hat{R}_{n}(#1, #2)}}

\newcommand{\hatRsub}[1]{\ensuremath{\hat{R}_{#1}}}
\newcommand{\hatRinf}{\ensuremath{\hatRsub{\infty}}}
\newcommand{\hatRsubij}[3]{\ensuremath{\hat{R}_{#1}(#2, #3)}}

\newcommand{\myState}{\ensuremath{\Omega}}

\newcommand{\StateAbs}{\ensuremath{\myState_{A}}}
\newcommand{\StateTrans}{\ensuremath{\myState_{T}}}

\newcommand{\norm}[1]{\left\lVert#1\right\rVert}

\newlength\figureheight 
\newlength\figurewidth 

\newtheorem{myLemma}{Lemma}

\newtheorem{mycor}{Corollary}

\newcounter{cnt}
\newcounter{mymagicrownumbers}

\newcommand{\Prob}{\ensuremath{\operatorname{Pr}}}

\usepackage{caption}
\usepackage[caption=false,font=footnotesize]{subfig}
\usepackage{url}

\begin{document}
%
\title{Markov Rewards Processes with Impulse Rewards and Absorbing States}
%
%


\author{\IEEEauthorblockN{Louis Tan, Kaveh Mahdaviani and Ashish Khisti}
\thanks{L.~Tan, K.~Mahdaviani and A.~Khisti are with the Dept.\ of Electrical and Computer Engineering, University of Toronto, Toronto, ON, Canada (e-mail: louis.tan@mail.utoronto.ca, mahdaviani@cs.toronto.edu, akhisti@ece.utoronto.ca). 
}%
}%
\maketitle

\begin{abstract}
%

We study the expected accumulated reward for a discrete-time Markov reward model with absorbing states. The rewards are impulse rewards, where a reward $\rew{i}{j}$ is accumulated when transitioning from state $i$ to state $j$.  We derive an explicit, single-letter expression for the expected accumulated reward as a function of the number of time steps $n$ and include in our analysis the limit in which $n \to \infty$.

\end{abstract}



%
\IEEEpeerreviewmaketitle

\section{Introduction}

Markov reward models have been a well-studied area of research for decades~\cite{howard1971dynamicv2} particularly in the literature for performance and dependability~\cite{Trivedi93, Muppala96, Blake1988,GayKetelsen1979,Beaudry78}.  Variations of the problem formulation have primarily been based on whether the Markov chain is discrete-time or continuous-time, whether there are any absorbing states in the state space, and whether rewards are assigned for the occupancy of a state (\emph{rate-based} Markov reward models) or for the transition to a state (\emph{impulse-based} Markov reward models).  Within any problem formulation, there have also been variations on the quantity of interest, with some authors calculating the expected instantaneous reward rate~\cite{Blake1988}, while others find the steady-state expected reward rate~\cite{GayKetelsen1979}, the expected accumulated reward~\cite{Blake1988}, the distribution of the accumulated reward until absorption~\cite{Beaudry78}, etc. (see~\cite{Trivedi93} for a review of the literature for different reward-based measures).  In terms of numerical computation, the topic of model checking for Markov chains has been used to verify whether certain properties hold such as~\cite{HanssonJonsson94}, ``after a request for service there is at least a 98\% probability that the service will be carried out within 2 seconds.''  Such a framework has also been extended for Markov reward models~\cite{KKZ05}.


Surprisingly, one formulation that has gone unstudied is that of finding the expected accumulated reward for a discrete-time Markov reward model with absorbing states and impulse rewards.  The continuous-time counterpart of this problem has been studied~\cite{Trivedi93}.  For a discrete-time model, to the author's knowledge, the results have either involved a steady-state analysis that excludes absorbing states~\cite{howard1971dynamicv2}, or a transient analysis for \emph{rate-based} models that include absorbing states but exclude impulse rewards~\cite{howard1971dynamicv2}.  


Granted, an impulse-based Markov reward model can be translated into a rate-based model by introducing an intermediary state between the transitioning states.  Specifically, suppose that in an impulse-based model, state~$i$ transitions to state~$j$ with probability $\trans{i}{j}$ and a reward of $\rew{i}{j}$ is assigned for such a transition.  Then in the rate-based counterpart to this model, for every such transition, we create an auxiliary state $k$ such that state~$i$ transitions to state~$k$ with probability~\trans{i}{j} and state~$k$ transitions to state~$j$ with probability one.  In this rate-based model, we assign the same reward~\rew{i}{j} for occupying state~$k$ and solve for the expected accumulated reward for rate-based models as in~\cite{howard1971dynamicv2} .  

While this approach is hypothetically possible, for a state space of size $|\myState|$, such an approach could potentially add an additional $|\myState|^2$ intermediate states if the states in the Markov model form a complete digraph.  As the solution in~\cite{howard1971dynamicv2} involves the inverting of a $|\myState| \times |\myState|$ matrix, the computational cost of this approach could be prohibitive.  In our work, we instead derive a closed-form solution for the expected accumulated reward without having to resort to introducing intermediary states.


\section{Problem Formulation}

A \emph{discrete-time Markov chain} is a sequence of random variables $\Si{0}, \Si{1}, \ldots $, where for every $i = 0, 1, \ldots$, state $\Si{i}$ takes values from the state space $\myState$, i.e., $\Si{i} \in \myState$, and the probability of transitioning from state $i$ at time $n-1$ to state $j$ at time $n$ given previous states $\Si{0}, \Si{1}, \ldots, \Si{n-1}$ satisfies the Markov property 

\begin{equation}
	\Prob(\Si{n} = j | \Si{n-1} = i, \Si{n- 2} = i_{n-2}, \ldots, \Si{1} = i_{1}) = \Prob(\Si{n} = j | \Si{n-1} = i).
\end{equation}
Without loss of generality, we assume the state space, $\myState$, is indexed by a set of integers so that $\myState = \{1, 2, \ldots, |\myState|\}$.

We study \emph{time-homogenous} Markov chains where the transition probabilities do not depend on $n$.  Specifically, at any time $n = 0, 1, \ldots$, the probability of transitioning from state $i$ to state $j$ does not depend on $n$, i.e., $\Prob(\Si{n} = j | \Si{n-1} = i) = \trans{i}{j}$.  At any time $n$, the transitions between states can therefore be described by a $|\myState| \times |\myState|$ \emph{transition matrix} $\transM$ whose $(i, j)$th entry is given by $\trans{i}{j}$ for $i, j \in \myState$.

In addition to being time-homogenous, the Markov chains we study are also \emph{absorbing}.

\begin{mydef}
	A state $i \in \myState$ is said to be \emph{absorbing} if for all $j \in \myState \setminus \{i\}$, $\trans{i}{j} = 0$  and $\trans{i}{i} = 1$.
\end{mydef}

\begin{mydef}
	A state is said to be \emph{transient} if it is not an absorbing state.
\end{mydef}

\begin{mydef}
	A discrete-time Markov chain is said to be an \emph{absorbing} Markov chain if it has at least one absorbing state, and it is possible to reach an absorbing state from any transient state within a finite number of transitions.
\end{mydef}

A \emph{Markov reward process} is a Markov chain that incorporates rewards that are accumulated during the evolution of the Markov chain.  The reward process we consider is additionally characterized by an \emph{impulse-reward matrix}.

\begin{mydef}
	The impulse-reward matrix, $\Rew$, is a $\sizeS \times \sizeS$ matrix whose $(i, j)th$ entry, $\rew{i}{j}$, represents the reward accumulated when transitioning from state $i$ to state $j$.
\end{mydef}
%
%
%
For any realization of a sequence of $n+1$ states from time step zero to time step $n$, i.e., given $\Si{0} = i_{0}$, $\Si{1} = i_{1}$, \ldots, $\Si{n} = i_{n}$, we define the accumulated reward $\Rn(\Si{0} = i_{0}, \Si{1} = i_{1}, \ldots, \Si{n} = i_{n})$ as 
\begin{equation}
	\label{eq:Rn}
	\Rn(\Si{0} = i_{0}, \Si{1} = i_{1}, \ldots, \Si{n} = i_{n}) = \sum_{k = 1}^{n} \rew{i_{k-1}}{i_{k}}.
\end{equation}

%


\begin{mydef}
\label{def:Rn}
	Let $\Rn$ be a random variable representing the accumulated reward at time step $n$ when the state sequence $\Si{0}, \Si{1}, \ldots, \Si{n}$ is taken to be random.   
\end{mydef}

\begin{mydef}
	\label{def:barRij}
	Let $\barRij{i}{j}$ be the expected value of $\Rn$ given initial state $\Si{0} = i$ and final state $\Si{n} = j$, i.e., $\barRij{i}{j} = \mathbb{E}(\Rn | \Si{0} = i, \Si{n} = j)$.  
\end{mydef}

From Definition~\ref{def:barRij}, the expression for $\barRij{i}{j}$ should be clear when $n = 1$.

\begin{mycor}
	\label{cor:barR1}
	$\barRsubij{1}{i}{j} = \rew{i}{j}$.
\end{mycor}

\begin{mydef}
	Let $\barR$ be the $|\myState| \times |\myState|$ matrix whose $(i, j)$th entry is given by $\barRij{i}{j}$.
\end{mydef}

\begin{mydef}
	\label{def:hatRij}
	Let $\hatRij{i}{j}$ be a scaled version of $\barRij{i}{j}$ where $\hatRij{i}{j} = \barRij{i}{j} \times \Prob(\Si{n} = j | \Si{0} = i)$.
\end{mydef}

\begin{mydef}
\label{def:hatR}
	Let $\hatR$ be the $|\myState| \times |\myState|$ matrix whose $(i, j)$th entry is given by $\hatRij{i}{j}$.
\end{mydef}

We are often also interested in the expected accumulated reward at time $n$ given only initial state $\Si{0} = i$ irrespective of the state at time $n$.

\begin{mydef}
\label{def:barRi}
	Let $\barRi{i}$ be the expected accumulated reward at time $n$, given initial state $\Si{0} = i$.
\end{mydef}

As we will see in the next section, the probability that the Markov reward process eventually reaches an absorbing state is unity, and so the long-term accumulated reward, also known as the reward until absorption, is of interest.

\begin{mydef}
\label{def:hatRinfij}
	Let $\hatRinf(i, j) = \lim_{n \to \infty} \hatRij{i}{j}$ be the long-term value of $\hatRij{i}{j}$.
\end{mydef}

\begin{mydef}
\label{def:hatRinf}
	Let $\hatRinf$ be the $|\myState| \times |\myState|$ matrix whose $(i, j)$th entry is given by $\hatRinf(i, j)$.
\end{mydef}

\begin{mydef}
\label{def:barRinfi}
	Let $\barRinf(i) = \lim_{n \to \infty} \barRi{i}$ be the long-term value of $\barRi{i}$.
\end{mydef}


In practice, $\barRinf(i)$ is often the most relevant quantity of interest, since we start the Markov reward process in an initial state and wish to know the accumulated reward before absorption.  We therefore define our problem as that of finding an expression for $\barRinf(i)$.  

We do this by first spending the majority of our efforts in Section~\ref{subsec:scaled_rewards} to derive an expression for the scaled accumulated reward variables, which culminates in the derivation of the transient scaled accumulated reward, $\hatR$, in~Lemma~\ref{lem:hatRn_matrix} and the long-term scaled accumulated reward, $\hatRinf$, in Theorem~\ref{thm:hatR_inf}.  In Section~\ref{subsec:unscaled_rewards}, we then show how to use these expressions to derive expressions for the \emph{unscaled} accumulated reward variables.  These variables are conditioned on initial state $\Si{0} = i$, and consist of the transient accumulated reward, $\barRi{i}$, in Theorem~\ref{thm:barRi}, and the accumulated reward before absorption, $\barRinf(i)$, in Corollary~\ref{cor:barRinfi}.  Finally, if a prior distribution over the initial states is known, we show how to calculate the unconditional transient accumulated reward, $\barR$, in Theorem~\ref{thm:barR_prior} and the unconditional accumulated reward before absorption, $\barRinf$, in Corollary~\ref{cor:barRinf}.


\section{Background for Absorbing Markov Chains}

A discrete-time absorbing Markov chain has a state space $\myState$ that can be partitioned into a set of absorbing states, $\StateAbs$, and a set of transient states, $\StateTrans$, such that $\myState = \StateAbs \cup \StateTrans$.  Recall our assumption that the state space is indexed by the integers $\{1, 2, \ldots, |\myState|\}$.  We further assume that transient states have \emph{lower} index values than absorbing states.  If this is the case, we can write out the transition matrix in its canonical form 

\begin{equation}
	\label{eq:transM}
	\transM = 
	\begin{bmatrix}
		\transQ & \transR \\ 
		0 & I
	\end{bmatrix},
\end{equation}
where
\begin{enumerate}
	\item $\transQ$ is a $|\StateTrans| \times |\StateTrans|$ matrix whose elements represent the probability of transitioning from one transient state to another transient state
	\item $\transR$ is a $|\StateTrans| \times |\StateAbs|$ matrix whose elements represent the probability of transitioning from a transient state to an absorbing state
	\item The matrix $0$ is the $|\StateAbs| \times |\StateTrans|$ zero matrix whose elements represent the impossibility of transitioning from an absorbing state to a transient state
	\item $I$ is the $|\StateAbs| \times |\StateAbs|$ identity matrix whose elements represent the probability of transitioning from one absorbing state to another absorbing state.
\end{enumerate}

\begin{remark}
	\label{rem:Rew}
	Given the assumption that transient states have \emph{lower} index values than absorbing states, we may also assume without loss of generality that the impulse-reward matrix, $\Rew$, has the form 
	\begin{equation}
	\label{eq:Rew_sub}
		\Rew = 
			\begin{bmatrix}
				\RewSubTransTrans & \RewSubTransAbs \\ 
				0 & 0
			\end{bmatrix},
	\end{equation}
	where 
	\begin{enumerate}
		\item $\RewSubTransTrans$ is a $|\StateTrans| \times |\StateTrans|$ matrix whose elements represent the reward accumulated for transitioning from one transient state to another transient state
		\item $\RewSubTransAbs$ is a $|\StateTrans| \times |\StateAbs|$ matrix whose elements represent the reward accumulated for transitioning from a transient state to an absorbing state
		\item the zero matrices have the appropriate dimensions for $\Rew$ to be a $\sizeS \times \sizeS$ matrix.
	\end{enumerate}
\end{remark}

\begin{mydef}
	\label{def:HtransRew}
	Let $\HtransRew$ be the Hadamard (element-wise) product of the reward matrix $\Rew$ and transition matrix $\transM$, i.e., $\HtransRew = \Rew \odot \transM$ so that $\HtransRewij{i}{j} = \rew{i}{j}\trans{i}{j}$.
\end{mydef}

\begin{remark}
	\label{rem:HtransRew}
	Similar to Remark~\ref{rem:Rew}, we may assume without loss of generality that~$\HtransRew$ has the form 
	\begin{equation}
	\label{eq:HtransRew_sub}
		\HtransRew = 
			\begin{bmatrix}
				\HSubTransTrans & \HSubTransAbs \\ 
				0 & 0
			\end{bmatrix}.
	\end{equation}
\end{remark}

From Definition~\ref{def:hatRij}, and Corollary~\ref{cor:barR1}, the expression for $\hatRij{i}{j}$ should be clear when $n = 1$.

\begin{mycor}
	\label{cor:hatR1}
	$\hatRsubij{1}{i}{j} = \HtransRewij{i}{j}$.
\end{mycor}

At time step $n$, the probability of being in state $j$ given initial state $i$ is given by the $(i, j)$th entry of $\transM^n$, which is given by the following lemma.

\begin{myLemma}
	\label{lem:transM_n}
	For $n = 1, 2, \ldots $, the transition matrix taken to the $n$th power is given by 
	\begin{equation}
	\label{eq:transM_n}
		\transM^n = 
			\begin{bmatrix}
				\transQ^n & \sum_{i = 0}^{n-1} \transQ^{i}\transR \\ 
				0 & I
			\end{bmatrix}.
	\end{equation}
\end{myLemma}

\begin{proof}
	We proceed by induction.
	\begin{LaTeXdescription}
		\item[Base case] It is readily verified by substituting $n = 1$ into \eqref{eq:transM_n} that we may recover the canonical form of the transition matrix in~\eqref{eq:transM}.
		\item[Inductive Hypothesis] We assume that for $n - 1 \geq 1$, 
			\begin{equation}
				\transM^{n-1} = 
					\begin{bmatrix*}[c]
						\transQ^{n-1} & \sum_{i = 0}^{n-2} \transQ^{i} \transR \\ 
						0 & I
					\end{bmatrix*}.
			\end{equation}
		\item[Induction Step] We have that
			\setcounter{cnt}{1}
			\begin{align}	
				\transM^n &= \transM^{n-1} \transM \\
				\label{eq:transM_n_mult}
				&\stackrel{(\alph{cnt})}{=} \begin{bmatrix*}[c]
						\transQ^{n-1} & \sum_{i = 0}^{n-2} \transQ^{i} \transR \\ 
						0 & I
					\end{bmatrix*}
					\begin{bmatrix}
						\transQ & \transR \\ 
						0 & I
					\end{bmatrix},%
			\end{align} where 
			\begin{enumerate}[(a)]
				\item follows from~\eqref{eq:transM} and the inductive hypothesis.
			\end{enumerate}
			After performing the block matrix multiplication in~\eqref{eq:transM_n_mult}, it is straightforward to see that the right-hand sides of~\eqref{eq:transM_n} and~\eqref{eq:transM_n_mult} are equal.
	\end{LaTeXdescription}
\end{proof}

\begin{myLemma}[{\hspace{1sp}\cite[Theorem 11.3]{GrinsteadSnell}}]
	\label{lem:transQn_zero}
	In an absorbing Markov chain, the probability that the process will be absorbed is 1~(i.e., $\transQ^n \to 0$ as $n \to \infty$).
\end{myLemma}

\begin{myLemma}[{\hspace{1sp}\cite[Theorem 11.4]{GrinsteadSnell}}]
	\label{lem:fund_mat}
	For an absorbing Markov chain, the matrix $I - \transQ$ has an inverse, $\fundMat$, termed the \emph{fundamental matrix}, and $\fundMat  = I+\transQ+\transQ^2 + \ldots$
\end{myLemma}

\begin{myLemma}
	\label{lem:transM_infty}
	The steady state probability $\transM^{\infty} = \lim_{n \to \infty} \transM^{n}$ is given by 
	\begin{equation}
	\label{eq:transM_infty}
		\transM^{\infty} = 
			\begin{bmatrix}
				0 & \fundMat \transR \\ 
				0 & I
			\end{bmatrix}.
	\end{equation}
\end{myLemma}
\begin{proof}
	This follows from substituting the results of Lemmas~\ref{lem:transQn_zero} and~\ref{lem:fund_mat} into the expression for $\transM^n$ in~\eqref{eq:transM_n}.
\end{proof}

\begin{mydef}[{\hspace{1sp}\cite[Definition 5.6.8]{HornJohnson}}]
	The \emph{spectral radius} $\rho(A)$ of a matrix $A \in \mathbb{R}^{n \times n}$ is 
	\begin{equation}
		\rho(A) \triangleq \max \{|\lambda| : \lambda \textrm{ is an eigenvalue of } A\}.
	\end{equation}
\end{mydef}

The spectral radius is itself not a matrix norm, however the following corollary states that there exists a norm that is arbitrarily close to the spectral radius.

\begin{myLemma}[{\hspace{1sp}\cite[Lemma 5.6.10]{HornJohnson}}]
	\label{lem:norm_bounds}
	Let $A \in \mathbb{R}^{n \times n}$ and $\epsilon > 0$ be given.  There is at least one matrix norm $\norm{\cdot}$ such that $\rho(A) \leq \norm{A} \leq \rho(A) + \epsilon$.
\end{myLemma}

\begin{myLemma}[{\hspace{1sp}\cite[Lemma 5.6.12]{HornJohnson}}]
	\label{lem:spectral}
	Let $A \in \mathbb{R}^{n \times n}$.  Then $\lim_{n \to \infty} A^n = 0$ if and only if $\rho(A) < 1$.
\end{myLemma}

By Lemmas~\ref{lem:transQn_zero} and~\ref{lem:spectral}, we have that $\rho(\transQ) < 1$.  By appropriately defining $\epsilon$ in Lemma~\ref{lem:norm_bounds}, it then follows that there is a matrix norm for which $\norm{\transQ} < 1$.

\begin{mycor}
	\label{cor:transQ_norm}
	For an absorbing markov chain, there exists a matrix norm for which $\norm{\transQ} < 1$.
\end{mycor}

\section{Expected Scaled Rewards}
\label{subsec:scaled_rewards}

We first derive Lemma~\ref{lem:barRij_property} to express $\barRij{i}{j}$ in terms of the elements in both the reward matrix and the transition matrix.  We then use this lemma to derive a recurrence relation for the scaled transient accumulated reward, $\hatR$, in~Lemma~\ref{lem:recurrence}.  We use the recurrence relation to derive an actual expression for the transient scaled reward in Lemma~\ref{lem:hatRn_matrix}.  Finally, we use properties of absorbing Markov chains to derive a single letter expression for the long-term scaled reward, $\hatRinf$, in Theorem~\ref{thm:hatR_inf}.

\begin{myLemma}
	\label{lem:barRij_property}
	Let $\barRij{i}{j}$ be defined as in Definition~\ref{def:barRij}.  For $n = 2, 3, \ldots$
	\begin{equation}
		\barRij{i}{j} = \chainsum \left( \rew{i}{k_1} + \rew{k_1}{k_2} + \ldots + \rew{k_{n-1}}{j}\right) \times \frac{\trans{i}{k_1}\trans{k_1}{k_2} \ldots \trans{k_{n-1}}{j}}{\Prob(\Si{n} = j| \Si{0} = i)}.
	\end{equation}
\end{myLemma}

\begin{proof}
	We calculate 
	\setcounter{cnt}{1}
	\begin{fleqn}
		\begin{alignat}{2}
			&\barRij{i}{j} \nonumber \\
			&\stackrel{(\alph{cnt})}{=} \mathbb{E}_{\Si{1}, \Si{2}, \ldots, \Si{n-1}} \mathbb{E}(\Rn | \Si{0} = i, \Si{1}, \Si{2}, \ldots, \Si{n-1}, \Si{n} = j) \\
				\addtocounter{cnt}{1}
			&= \chainsum  \Big\{ \mathbb{E}(\Rn | \Si{0} = i, \Si{1} = k_1, \Si{2} = k_2, \ldots, \Si{n-1} = k_{n-1},\Si{n} = j) \phantom{\Big\}} \nonumber \\
				& \hphantom{=\chainsum} \qquad \times \vphantom{\Big\{} \Prob(\Si{1} = k_1, \Si{2} = k_2, \ldots, \Si{n-1} = k_{n-1}| \Si{0} = i, \Si{n} = j) \Big\} \\
			&\stackrel{(\alph{cnt})}{=} \chainsum  \left\{ \mathbb{E}(\Rn | \Si{0} = i, \Si{1} = k_1, \Si{2} = k_2, \ldots, \Si{n-1} = k_{n-1},\Si{n} = j) \vphantom{\frac{\Prob\Si{1}}{\Prob\Si{2}}} \right. \nonumber \\
				&\hphantom{=\chainsum} \qquad \times \left. \frac{\Prob(\Si{1} = k_1, \Si{2} = k_2, \ldots, \Si{n} = j | \Si{0} = i)}{\Prob(\Si{n} = j| \Si{0} = i)} \right\} \\
		\addtocounter{cnt}{1}
			&\stackrel{(\alph{cnt})}{=} \chainsum \left( \rew{i}{k_1} + \rew{k_1}{k_2} + \ldots + \rew{k_{n-1}}{j}\right) \times \frac{\Prob(\Si{1} = k_1, \Si{2} = k_2, \ldots, \Si{n} = j | \Si{0} = i)}{\Prob(\Si{n} = j| \Si{0} = i)} \\
		\addtocounter{cnt}{1}
			&\stackrel{(\alph{cnt})}{=} \chainsum \left( \rew{i}{k_1} + \rew{k_1}{k_2} + \ldots + \rew{k_{n-1}}{j}\right) \times \frac{\trans{i}{k_1}\trans{k_1}{k_2} \ldots \trans{k_{n-1}}{j}}{\Prob(\Si{n} = j| \Si{0} = i)} 
		\end{alignat}
	\end{fleqn}
	where 
	\begin{enumerate}[(a)]
		\item follows from Definition~\ref{def:barRij} and the law of total expectation
		\item follows from Bayes' Theorem
		\item follows from the fact that we have conditioned on each state from time $t = 0, 1, 2, \ldots, n$, and so the additive rewards for each transition is known (see~\eqref{eq:Rn})
		\item follows from the Markov property
	\end{enumerate}
\end{proof}

\begin{myLemma}
	Let $\hatR$ be defined as in Definition~\ref{def:hatR}.  For the case of $n = 2$, 
	\begin{equation}
		\label{eq:hatR2}
		\hatRsub{2} = \HtransRew \transM + \transM \HtransRew,
	\end{equation}
	where $\HtransRew$ is defined in Definition~\ref{def:HtransRew} and $\transM$ is given by~\eqref{eq:transM}.
\end{myLemma}
\begin{proof}
	\setcounter{cnt}{1}
	We use Definition~\ref{def:hatRij} and substitute $n = 2$ into Lemma~\ref{lem:barRij_property} to get 
	\begin{align}
		\hatRsubij{2}{i}{j} &= \sum_{k_1 = 1}^{\sizeS} (\rew{i}{k_1} + \rew{k_1}{j}) \trans{i}{k_1}\trans{k_1}{j} \\
		&= \sum_{k_1 = 1}^{\sizeS}(\rew{i}{k_1}\trans{i}{k_1}) \trans{k_1}{j} +  \trans{i}{k_1}(\rew{k_1}{j} \trans{k_1}{j} ) \\
		&\stackrel{(\alph{cnt})}{=} \sum_{k_1 = 1}^{\sizeS}\HtransRewij{i}{k_1} \trans{k_1}{j} +  \trans{i}{k_1}\HtransRewij{k_1}{j} \\	
		\addtocounter{cnt}{1}		
		&\stackrel{(\alph{cnt})}{=} [\HtransRew \transM]_{i, j} +  [\transM \HtransRew]_{i, j} 	
	\end{align}
	where 
\begin{enumerate}[(a)]
	\item follows from Definition~\ref{def:HtransRew}
	\item follows from the definition of matrix multiplication
\end{enumerate}
\end{proof}


We now use Lemma~\ref{lem:barRij_property} to derive a recurrence relation for $\hatR$, the scaled transient accumulated reward at time step $n$.

\begin{myLemma}
\label{lem:recurrence}
	Let $\hatR$ be defined as in Definition~\ref{def:hatR}.  For $n = 2, 3,  \ldots $
	\begin{equation}
	\label{eq:lem_recurrence}
		\hatR = \hatRsub{n-1} \transM + \transMsup{n-1}\HtransRew,
	\end{equation}
	where $\HtransRew$ is defined in Definition~\ref{def:HtransRew} and $\transM$ is given by~\eqref{eq:transM}.
\end{myLemma}

\begin{proof}
We first prove the lemma for $n = 3, 4, \ldots $
\setcounter{cnt}{1}
\begin{fleqn}
	\begin{alignat}{2}
		&\hatRij{i}{j} \nonumber \\
			&\stackrel{(\alph{cnt})}{=} \sum_{k_{n-1}=1}^{\sizeS} \left( \left\{\chainsumOne \left( \rew{i}{k_1} + \rew{k_1}{k_2} + \ldots + \rew{k_{n-2}}{k_{n-1}}\right) \times \trans{i}{k_1}\trans{k_1}{k_2} \ldots \trans{k_{n-2}}{k_{n-1}}\right\} \right.  \\
			&\hphantom{=\sum_{k_{n-1} = 1}^{\sizeS} } \qquad \times \trans{k_{n-1}}{j}  + \left.  \chainsumOne  \rew{k_{n-1}}{j} \times \trans{i}{k_1}\trans{k_1}{k_2} \ldots \trans{k_{n-1}}{j}\right) \\
		\addtocounter{cnt}{1}
			&\stackrel{(\alph{cnt})}{=} \sum_{k_{n-1}=1}^{\sizeS} \left(  \Big\{ \hatRsubij{n-1}{i}{k_{n-1}} \Big\} \trans{k_{n-1}}{j} +\chainsumOne  \rew{k_{n-1}}{j} \times \trans{i}{k_1}\trans{k_1}{k_2} \ldots \trans{k_{n-1}}{j} \right) \\
		\addtocounter{cnt}{1}
			&= \sum_{k_{n-1}=1}^{\sizeS} \left(   \hatRsubij{n-1}{i}{k_{n-1}}  \trans{k_{n-1}}{j} + \rew{k_{n-1}}{j} \trans{k_{n-1}}{j}\chainsumOne   \trans{i}{k_1}\trans{k_1}{k_2} \ldots \trans{k_{n-2}}{k_{n-1}} \right) \\
			&\stackrel{(\alph{cnt})}{=} \sum_{k_{n-1}=1}^{\sizeS}   \hatRsubij{n-1}{i}{k_{n-1}}  \trans{k_{n-1}}{j} + \rew{k_{n-1}}{j} \trans{k_{n-1}}{j}\transMsupij{n-1}{i}{k_{n-1}} \\
		\addtocounter{cnt}{1}
			&\stackrel{(\alph{cnt})}{=} \sum_{k_{n-1}=1}^{\sizeS}   \hatRsubij{n-1}{i}{k_{n-1}}  \trans{k_{n-1}}{j} + \HtransRewij{k_{n-1}}{j} \transMsupij{n-1}{i}{k_{n-1}} \\
		\addtocounter{cnt}{1}
			&\stackrel{(\alph{cnt})}{=}  [\hatRsub{n-1}  \transM]_{i,j} +  [\transMsup{n-1} \HtransRew ]_{i, j}
	\end{alignat}
\end{fleqn}
where 
\begin{enumerate}[(a)]
	\item follows from Definition~\ref{def:hatRij} and rearranging Lemma~\ref{lem:barRij_property}
	\item follows from Definition~\ref{def:hatRij} and the application of Lemma~\ref{lem:barRij_property} for $\barRsubij{n-1}{i}{k_{n-1}}$
	\item follows from Lemma~\ref{lem:transMsup} in Appendix~\ref{sec:matrix_properties}, where we have used the fact that $n \geq 3$
	\item follows from Definition~\ref{def:HtransRew}
	\item follows from the definition of matrix multiplication
\end{enumerate}

We mention that although we derived the lemma assuming $n \in \{3, 4, \ldots\}$, the lemma also holds if $n=2$.  We can see this by using  Corollary~\ref{cor:hatR1} to compare the right-hand-sides of \eqref{eq:hatR2} and \eqref{eq:lem_recurrence} when $n=2$.

\end{proof}

We next use the \emph{recurrence} relation for $\hatR$ in Lemma~\ref{lem:recurrence} to derive an \emph{explicit} expression for $\hatR$, the scaled transient accumulated reward at time step $n$.

\begin{myLemma}
	\label{lem:hatRn_matrix}
	Let $\hatR$ be defined as in Definition~\ref{def:hatR}.
	For $n = 1, 2, \ldots $
	\begin{equation}
	\label{eq:hatRn_matrix}
		\hatR = 
		\begin{bmatrix}
			\RewSubD & \RewSubC\\
			0 & 0 
		\end{bmatrix},
	\end{equation}	
	where
	\begin{align}
		\label{eq:RewSubDn}
			\RewSubD &= \sum_{i = 0}^{n-1}\transQ^{i}\HSubTransTrans\transQ^{n - i - 1}, \\
		\label{eq:RewSubCn}
			\RewSubC &= \fundMat(I - \transQ^{n})\HSubTransAbs + \fundMat(I - \transQ^{n-1})\HSubTransTrans\fundMat\transR - \sum_{i = 0}^{n-2}\transQ^{i}\HSubTransTrans\fundMat\transQ^{n - i - 1}\transR,
	\end{align}
	the matrices $\HSubTransTrans$ and $\HSubTransAbs$ are given in Remark~\ref{rem:HtransRew}, $\fundMat$ is the fundamental matrix of Lemma~\ref{lem:fund_mat} and the matrices $\transQ$ and $\transR$ are defined in~\eqref{eq:transM}.
\end{myLemma}

\begin{proof}
	We proceed by induction.
	\begin{LaTeXdescription}
		\item[Base case] We use Corollary~\ref{cor:hatR1} and Remark~\ref{rem:HtransRew} to verify \eqref{eq:hatRn_matrix} for the base case after substituting $n = 1$ into~\eqref{eq:RewSubDn}  and~\eqref{eq:RewSubCn} to get that 
		\setcounter{cnt}{1}
		\begin{align}
			\RewSubDi{1} &= \HSubTransTrans, \\		
			\RewSubCi{1} &= \fundMat(I - \transQ)\HSubTransAbs \\
			&\stackrel{(\alph{cnt})}{=} \HSubTransAbs,
		\end{align} %
		where 
		\begin{enumerate}[(a)]
			\item follows from the definition of the fundamental matrix in Lemma~\ref{lem:fund_mat}.
		\end{enumerate}
		
		\item[Inductive Hypothesis] We assume that for $n - 1 \geq 1$, 
			\begin{align}
				\RewSubDi{n-1} &= \sum_{i = 0}^{n-2}\transQ^{i}\HSubTransTrans\transQ^{n - i - 2}, \\			
				\RewSubCi{n-1} &= \fundMat(I - \transQ^{n-1})\HSubTransAbs + \fundMat(I - \transQ^{n-2})\HSubTransTrans\fundMat\transR - \sum_{i = 0}^{n-3}\transQ^{i}\HSubTransTrans\fundMat\transQ^{n - i - 2}\transR.
			\end{align}
		
		\item[Induction Step] We begin with the recurrence relation in Lemma~\ref{lem:recurrence} from which we get 
			\setcounter{cnt}{1}
			\begin{align}
				\hatR &= \hatRsub{n-1} \transM + \transMsup{n-1}\HtransRew \\
				&\stackrel{(\alph{cnt})}{=} 	
					\begin{bmatrix}
						\RewSubDi{n-1} & \RewSubCi{n-1}\\
						0 & 0
					\end{bmatrix}
					P + \transMsup{n-1}\HtransRew \\
				\addtocounter{cnt}{1}
				&\stackrel{(\alph{cnt})}{=} 	
					\begin{bmatrix}
						\RewSubDi{n-1} & \RewSubCi{n-1}\\
						0 & 0
					\end{bmatrix}
					\begin{bmatrix}
						\transQ & \transR \\ 
						0 & I
					\end{bmatrix} + 
					\begin{bmatrix}
						\transQ^{n-1} & \sum_{i = 0}^{n-2} \transQ^{i}\transR \\ 
						0 & I
					\end{bmatrix} H\\
				\addtocounter{cnt}{1}
				&\stackrel{(\alph{cnt})}{=}
					\label{eq:before_block}
					\begin{bmatrix}
						\RewSubDi{n-1} & \RewSubCi{n-1}\\
						0 & 0
					\end{bmatrix}
					\begin{bmatrix}
						\transQ & \transR \\ 
						0 & I
					\end{bmatrix} + 
					\begin{bmatrix}
						\transQ^{n-1} & \sum_{i = 0}^{n-2} \transQ^{i}\transR \\ 
						0 & I
					\end{bmatrix}
					\begin{bmatrix}
						\HSubTransTrans & \HSubTransAbs\\ 
						0 & 0
					\end{bmatrix} 
				\addtocounter{cnt}{1}					
			\end{align}
			where 
			\begin{enumerate}[(a)]
				\item follows from the inductive hypothesis
				\item follows from Lemma~\ref{lem:transM_n}
				\item follows from Remark~\ref{rem:HtransRew}.
			\end{enumerate}

			We first consider $\RewSubDi{n}$.  We perform the block matrix multiplication in~\eqref{eq:before_block} to get that 
			\begin{align}
				\setcounter{cnt}{1}
				\RewSubDi{n} &=	\RewSubDi{n-1} \transQ + \transQ^{n-1}\HSubTransTrans \\
				&\stackrel{(\alph{cnt})}{=} \left( \sum_{i = 0}^{n-2}\transQ^{i}\HSubTransTrans\transQ^{n - i - 2} \right) \transQ +  \transQ^{n-1}\HSubTransTrans\\
				\addtocounter{cnt}{1}	
				&=  \sum_{i = 0}^{n-2}\transQ^{i}\HSubTransTrans\transQ^{n - i - 1}+  \transQ^{n-1}\HSubTransTrans\\
				\label{eq:RewSubDi_final}
				&=  \sum_{i = 0}^{n-1}\transQ^{i}\HSubTransTrans\transQ^{n - i - 1},
			\end{align}
			where 
			\begin{enumerate}[(a)]
				\item follows from the inductive hypothesis.
			\end{enumerate}
			We conclude that~\eqref{eq:RewSubDn} holds after comparing with~\eqref{eq:RewSubDi_final} in the induction step. Next, for $\RewSubCi{n}$, we again perform the block matrix multiplication in~\eqref{eq:before_block} to get that 
			\begin{align}
				\setcounter{cnt}{1}
				\RewSubCi{n} &= \RewSubDi{n-1} \transR +  \RewSubCi{n-1}  + \transQ^{n-1}\HSubTransAbs \\
				&\stackrel{(\alph{cnt})}{=} \left( \sum_{i = 0}^{n-2}\transQ^{i}\HSubTransTrans\transQ^{n - i - 2} \right) \transR + \left( \fundMat(I - \transQ^{n-1})\HSubTransAbs + \fundMat(I - \transQ^{n-2})\HSubTransTrans\fundMat\transR \vphantom{\sum_{i=0}^{n}} \right. \\ \nonumber
				&\qquad \left. - \sum_{i = 0}^{n-3}\transQ^{i}\HSubTransTrans\fundMat\transQ^{n - i - 2}\transR \right) + \transQ^{n-1}\HSubTransAbs \\
				\addtocounter{cnt}{1}				
				&\stackrel{(\alph{cnt})}{=}  \left( \sum_{i = 0}^{n-2}\transQ^{i}\HSubTransTrans \fundMat (I - \transQ)\transQ^{n - i - 2}  \transR  - \sum_{i = 0}^{n-3}\transQ^{i}\HSubTransTrans\fundMat\transQ^{n - i - 2}\transR \right) \\ \nonumber
				&\qquad +  \fundMat(I - \transQ^{n-1})\HSubTransAbs + \fundMat(I - \transQ^{n-2})\HSubTransTrans\fundMat\transR  + \transQ^{n-1}\HSubTransAbs \\
				\addtocounter{cnt}{1}				
				&=  \left( \sum_{i = 0}^{n-2} \left( \transQ^{i}\HSubTransTrans \fundMat \transQ^{n - i - 2}  \transR -  \transQ^{i}\HSubTransTrans \fundMat\transQ^{n - i - 1} \transR \right) - \sum_{i = 0}^{n-3}\transQ^{i}\HSubTransTrans\fundMat\transQ^{n - i - 2}\transR \right) \\ \nonumber 
				&\qquad +  \fundMat(I - \transQ^{n-1})\HSubTransAbs + \fundMat(I - \transQ^{n-2})\HSubTransTrans\fundMat\transR + \transQ^{n-1}\HSubTransAbs \\
				&=  \left(   \transQ^{n-2} \HSubTransTrans  \fundMat  \transR -  \sum_{i = 0}^{n-2} \transQ^{i}\HSubTransTrans \fundMat\transQ^{n - i - 1} \transR \right) \\ \nonumber
				&\qquad +  \fundMat(I - \transQ^{n-1})\HSubTransAbs + \fundMat(I - \transQ^{n-2})\HSubTransTrans\fundMat\transR + \transQ^{n-1}\HSubTransAbs \\
				&=   \fundMat(I - \transQ^{n-1} + \fundMat^{-1}\transQ^{n-1})\HSubTransAbs + \fundMat(I - \transQ^{n-2} + \fundMat^{-1}\transQ^{n-2})\HSubTransTrans\fundMat\transR  \\ \nonumber
				&\qquad -  \sum_{i = 0}^{n-2} \transQ^{i}\HSubTransTrans \fundMat\transQ^{n - i - 1} \transR   \\
				\label{eq:RewSubCn_final}
				&\stackrel{(\alph{cnt})}{=}\fundMat(I - \transQ^{n})\HSubTransAbs + \fundMat(I - \transQ^{n-1})\HSubTransTrans\fundMat\transR  -  \sum_{i = 0}^{n-2} \transQ^{i}\HSubTransTrans \fundMat\transQ^{n - i - 1} \transR,
			\end{align}
			where 
			\begin{enumerate}[(a)]
				\item follows from the inductive hypothesis
				\item and (c) follow from the definition of the fundamental matrix, i.e., $\fundMat^{-1} = (I - \transQ)$ in Lemma~\ref{lem:fund_mat}.
			\end{enumerate}
			We conclude the proof after comparing~\eqref{eq:RewSubCn} with~\eqref{eq:RewSubCn_final} in the induction step.
	\end{LaTeXdescription}
\end{proof}


Finally, after having found an explicit expression for $\hatR$ at time step $n$, we now use properties of absorbing Markov chains to derive an expression for the long-term scaled reward, $\hatRinf$.

\begin{theorem}
	\label{thm:hatR_inf}
	Let $\hatRinf$ be defined as in Definition~\ref{def:hatRinf}. Let $\RewSubDi{\infty} = \lim_{n \to \infty} \RewSubD$, and $\RewSubCi{\infty} = \lim_{n \to \infty} \RewSubC$ where $\RewSubD$ and $\RewSubC$ are given by~\eqref{eq:RewSubDn} and~\eqref{eq:RewSubCn} respectively.  Then $\RewSubDi{\infty} = 0$ and 
	\begin{equation}
	\label{eq:hatRnInf_matrix}
		\hatRinf = 
		\begin{bmatrix}
			0 & \RewSubCi{\infty} \\
			0 & 0 
		\end{bmatrix},
	\end{equation}	
	where
	\begin{align}
		\label{eq:RewSubCInf}
		\RewSubCi{\infty} &= \fundMat(\HSubTransAbs + \HSubTransTrans\fundMat\transR),
	\end{align}
	the matrices $\HSubTransTrans$ and $\HSubTransAbs$ are given in Remark~\ref{rem:HtransRew}, $\fundMat$ is the fundamental matrix of Lemma~\ref{lem:fund_mat} and $\transR$ is defined in~\eqref{eq:transM}.
\end{theorem}

\begin{proof}	
	We begin by writing
	\setcounter{cnt}{1}
	\begin{align}
		\RewSubCi{\infty} &\stackrel{(\alph{cnt})}{=} \lim_{n \to \infty} \left( \fundMat(I - \transQ^{n})\HSubTransAbs + \fundMat(I - \transQ^{n-1})\HSubTransTrans\fundMat\transR - \sum_{i = 0}^{n-2}\transQ^{i}\HSubTransTrans\fundMat\transQ^{n - i - 1}\transR\right) \\
		\addtocounter{cnt}{1}
		\label{eq:RewSubCInf_summation}
		&\stackrel{(\alph{cnt})}{=}  \fundMat\HSubTransAbs + \fundMat\HSubTransTrans\fundMat\transR - \lim_{n \to \infty}  \sum_{i = 0}^{n-2}\transQ^{i}\HSubTransTrans\fundMat\transQ^{n - i - 1}\transR,
		\addtocounter{cnt}{1}
	\end{align}
	where 
	\begin{enumerate}[(a)]
		\item follows from Lemma~\ref{lem:hatRn_matrix}
		\item follows from Lemma~\ref{lem:transQn_zero}.
	\end{enumerate}
	In order to prove~\eqref{eq:RewSubCInf}, we must now show the last term in~\eqref{eq:RewSubCInf_summation} converges to the zero matrix.  We use Corollary~\ref{cor:transQ_norm} and Lemma~\ref{lem:summation_matrix_zero} in Appendix~\ref{sec:matrix_properties} to conclude that this is indeed the case.  
	
	For $\RewSubDi{\infty}$, we have that 
	\begin{align}
		\RewSubDi{\infty} &= \lim_{n \to \infty} \sum_{i = 0}^{n-1}\transQ^{i}\HSubTransTrans\transQ^{n - i - 1} \\
		\label{eq:RewSubDi_inf_proof}
		&= \lim_{n \to \infty} \left( \transQ^{n-1} \HSubTransTrans + \sum_{i = 0}^{n-2}\transQ^{i}\HSubTransTrans\transQ^{n - i - 1} \right)
		\addtocounter{cnt}{1}.
	\end{align}
	We conclude from Lemma~\ref{lem:transQn_zero} that the first term in~\eqref{eq:RewSubDi_inf_proof} converges to the zero matrix.  Finally we again use Corollary~\ref{cor:transQ_norm} and Lemma~\ref{lem:summation_matrix_zero} in Appendix~\ref{sec:matrix_properties} to conclude that the second term in~\eqref{eq:RewSubDi_inf_proof} also converges to the zero matrix so that $\RewSubDi{\infty} = 0$.
\end{proof}

\section{Expected Unscaled Rewards}
\label{subsec:unscaled_rewards}





Having found an expression for $\hatR$ in the previous section, we now show that given initial state $\Si{0} = i$, the expected accumulated reward at time $n$ is given by the sum over all columns of the $i$th row of the matrix $\hatR$.

\begin{theorem}
\label{thm:barRi}
	Let $\barRi{i}$ be defined as in Definition~\ref{def:barRi}.  Then
\begin{align}
	\barRi{i} &=  \sum_{j = 1}^{|\myState|} \hatRij{i}{j},
\end{align}
where $\hatRij{i}{j}$ is defined as in Definition~\ref{def:hatRij} and is given by Lemma~\ref{lem:hatRn_matrix}.
\end{theorem}
\begin{proof}
	By the law of total expectation we have that 
	\setcounter{cnt}{1}
	\begin{align}
		\barRi{i} &= \sum_{j = 1}^{|\myState|} \barRij{i}{j} \Prob(\Si{n} = j | \Si{0} = i).
	\end{align}
	We conclude the result in Theorem~\ref{thm:barRi} holds by Definition~\ref{def:hatRij}.
\end{proof}

\begin{mycor}
\label{cor:barRinfi}
	Let $\barRinf(i)$ be defined as in Definition~\ref{def:barRinfi}.  Then
\begin{align}
	\barRinf(i) &=  \sum_{j = 1}^{|\myState|} \hatRinf(i, j),
\end{align}
where $\hatRinf(i, j)$ is defined as in Definition~\ref{def:hatRinfij} and is given by Theorem~\ref{thm:hatR_inf}.
\end{mycor}

Similarly, given $\barRi{i}$ and a prior distribution over initial states, we can use the law of total expectation to calculate the unconditional expected value of $\Rn$. 

\begin{theorem}
\label{thm:barR_prior}
	Let $\Rn$ be defined as in Definition~\ref{def:Rn}. If a prior distribution $\Prob(\Si{0})$ over the initial state $\Si{0}$ is known, then $\mathbb{E}\Rn$, the expected value of $\Rn$ is given by
	\begin{equation}
		\mathbb{E}\Rn = \sum_{i = 1}^{|\myState|} \barRi{i} \Prob(\Si{0} = i),
	\end{equation}
where $\barRi{i}$ is defined as in Definition~\ref{def:barRi} and is given by Theorem~\ref{thm:barRi}.
\end{theorem}
\begin{mycor}
	Let $\barRinf = \lim_{n \to \infty} \mathbb{E}\Rn$.  Then 
\label{cor:barRinf}
\begin{align}
	\barRinf &=  \sum_{i= 1}^{|\myState|} \barRinf(i)\Prob(\Si{0} = i),
\end{align}
where $\barRinf(i)$ is defined as in Definition~\ref{def:barRinfi} and is given by Corollary~\ref{cor:barRinfi}.
\end{mycor}

Finally, it may be of interest to know the expected accumulated reward after absorption given initial state~$i$ and absorbing state~$j$.

\begin{theorem}
\label{thm:barRinf_absorb}
	Let $i \in \StateTrans$ and $j \in \StateAbs$.  Let $\barRinf(i, j)$ represent the expected accumulated reward after absorption given initial state~$i$ and absorbing state~$j$.  Then 
	\begin{align}
		\barRinf(i, j) &=  \frac{1}{\transM^{\infty}(i, j)}\hatRinf(i, j), 
	\end{align}
	where $\transM^{\infty}(i, j)$ is the $(i, j)th$ entry of $\transM^{\infty}$ given in Lemma~\ref{lem:transM_infty}.
\end{theorem}


\appendices

%
%
%
%

\section{Proof of Supporting Lemmas}
\label{sec:matrix_properties}

\begin{myLemma}
	\label{lem:transMsup}
	Let $A$ be an $m \times m$ matrix, and let $A^n$ denote the $n$th power of $A$  for some $n \in \{2, 3, \ldots \}$.  Then the $(i, j)$th entry of $A^n$, denoted by $A^n(i, j)$, is given by 
	\begin{equation}
	\label{eq:transMsup}
		A^n(i, j) = \sum_{k_1 = 1}^m \sum_{k_2 = 1}^m \ldots \sum_{k_{n-1} = 1}^m A(i, k_1)A(k_1, k_2) \ldots A(k_{n-1}, j).
	\end{equation}
\end{myLemma}

\begin{proof}
	We proceed by induction.
	\begin{LaTeXdescription}
		\item[Base case] We substitute $n = 2$ into \eqref{eq:transMsup} to get that 
			\begin{equation}
				A^2(i, j) = \sum_{k_1 = 1}^m A(i, k_1)A(k_1, j),
			\end{equation}
			which is the familiar definition for matrix multiplication.
		\item[Inductive Hypothesis] We assume that for $n - 1 \geq 2$, 
			\begin{equation}
				A^{n-1}(i, j) = \sum_{k_1 = 1}^m \sum_{k_2 = 1}^m \ldots \sum_{k_{n-2} = 1}^m A(i, k_1)A(k_1, k_2) \ldots A(k_{n-2}, j).
			\end{equation}
		\item[Induction Step] We have that
			\setcounter{cnt}{1}
			\begin{align}
				A^{n}(i, j) &\stackrel{(\alph{cnt})}{=} \sum_{k_{n-1} = 1}^m A^{n-1}(i, k_{n-1}) A(k_{n-1}, j) \\
				\addtocounter{cnt}{1}		
				\label{eq:transMsup_final}
				&\stackrel{(\alph{cnt})}{=} \sum_{k_{n-1} = 1}^m \left\{\sum_{k_1 = 1}^m \sum_{k_2 = 1}^m \ldots \sum_{k_{n-2} = 1}^m A(i, k_1)A(k_1, k_2) \ldots A(k_{n-2}, k_{n-1}) \right\} A(k_{n-1}, j),
			\end{align}
			where 
			\begin{enumerate}[(a)]
				\item follows from the definition of matrix multiplication and the fact that $A^n = A^{n-1} A$
				\item follows from the inductive hypothesis.
			\end{enumerate}
			After rearranging the right-hand-side of~\eqref{eq:transMsup_final}, we conclude that~\eqref{eq:transMsup} indeed holds.
	\end{LaTeXdescription}
\end{proof}


\begin{myLemma}
	\label{lem:summation_matrix_zero}
	Let $g:  \mathbb{N} \times \mathbb{R}^{m \times m} \times \mathbb{R}^{m \times m} \mapsto \mathbb{R}^{m \times m}$ be the function given by 
	\begin{equation}
	g(n, A, Q) = \sum_{i = 0}^{n-2}\transQ^{i}A\transQ^{n - i - 1}.
	\end{equation}
	where $\transQ$ is a matrix such that there exists a norm for which $\norm{Q} < 1$. Then 
	\begin{equation}
	\lim_{n \to \infty} g(n, A, Q) = 0_{m \times m}.
	\end{equation}
\end{myLemma}
\begin{proof}
	We show that the norm of $g(n, A, Q)$ approaches zero as $n \to \infty$.  For any $n$, we have that
	\setcounter{cnt}{1}
	\begin{align}
		\norm{g(n, A, Q)} &= \norm{\sum_{i = 0}^{n-2} \transQ^{i}A\transQ^{n - i - 1}}\\
		&\stackrel{(\alph{cnt})}{\leq} \sum_{i = 0}^{n-2} \norm{\transQ^{i} A \transQ^{n - i - 1}} \\
		\addtocounter{cnt}{1}
		&\stackrel{(\alph{cnt})}{\leq}  \sum_{i = 0}^{n-2} \norm{\transQ}^{i} \cdot \norm{A} \cdot \norm{\transQ}^{n - i - 1} \\
		\addtocounter{cnt}{1}
		&= \sum_{i = 0}^{n-2} \norm{\transQ}^{n - 1} \cdot \norm{A} \\
		&= (n-2) \norm{\transQ}^{n - 1} \cdot \norm{A},
		\label{eq:lHospital}
	\end{align}
	where 
	\begin{enumerate}[(a)]
		\item follows from sub-additive property of the matrix norm
		\item follows from sub-multiplicative property of the matrix norm.
	\end{enumerate}
	Finally, we use L'Hospital's Rule and the assumption that $\norm{Q} < 1$ to conclude that the right-hand-side of~\eqref{eq:lHospital} approaches zero as $n \to \infty$.
\end{proof}

\ifCLASSOPTIONcaptionsoff
  \newpage
\fi



%
%
%
%

\vspace{-1em}

\bibliographystyle{IEEEtran}
\bibliography{IEEEabrv,itw2013emina2}

\vfill


\end{document}